\documentclass{amsart}
\usepackage{amsmath, amssymb,epic,graphicx,mathrsfs}
\usepackage[all]{xy}
\usepackage{color}
\usepackage{comment}
\usepackage{tikz}

\usepackage{url}

\usepackage{amsthm}
\usepackage{amssymb}
\usepackage{latexsym}
\usepackage{longtable}
\usepackage{epsfig}
\usepackage{amsmath}
\usepackage{hhline}
\numberwithin{equation}{section}

\DeclareMathOperator{\frat}{Frat}

\newcommand{\nn}{\mathrel{\unlhd}}
\newcommand{\gen}[1]{\left\langle#1\right\rangle}
 \DeclareMathOperator{\perm}{Sym}

\DeclareMathOperator{\stab}{Stab}

\DeclareMathOperator{\gl}{GL} 
\DeclareMathOperator{\pg}{PG}

\DeclareMathOperator{\End}{End}

\newcommand{\rg}{\Gamma}

\renewcommand{\emptyset}{\varnothing}

\DeclareMathOperator{\psl}{PSL}
\newcommand{\st}{such that }
\newcommand{\ifa}{if and only if }

\newtheorem{thm}{Theorem}
\newtheorem{cor}[thm]{Corollary}
 \newtheorem{lemma}[thm]{Lemma}
\newtheorem{prop}[thm]{Proposition} \newtheorem{rem}[thm]{Remark}
 
 \newtheorem{con}[]{Conjecture}

\begin{document}
	\bibliographystyle{amsplain}
\title[Forbidden subgraphs in generating graphs]{Forbidden subgraphs in
	generating graphs\\ of finite groups}

\author{Andrea Lucchini}
\address{Andrea Lucchini\\ Universit\`a degli Studi di Padova\\  Dipartimento di Matematica \lq\lq Tullio Levi-Civita\rq\rq\\ Via Trieste 63, 35121 Padova, Italy\\email: lucchini@math.unipd.it}
\author{Daniele Nemmi}
\address{Daniele Nemmi\\ Universit\`a degli Studi di Padova\\  Dipartimento di Matematica \lq\lq Tullio Levi-Civita\rq\rq\\ Via Trieste 63, 35121
	Padova, Italy\\email: daniele.nemmi@phd.unipd.it}

	\begin{abstract} Let $G$ be a $2$-generated group. The generating graph $\Gamma(G)$ is the graph whose vertices
	are the elements of $G$ and where two vertices $g_1$ and $g_2$ are adjacent if $G = \langle g_1, g_2 \rangle.$ This graph
	encodes the combinatorial structure of the distribution of generating pairs across $G.$ In this
	paper we study 
	some graph theoretic properties of $\Gamma(G)$, with particular emphasis on those properties that can be formulated in terms of forbidden induced subgraphs. In particular we investigate when the generating graph $\Gamma(G)$ is a cograph (giving a complete description when $G$ is soluble) and when it is perfect (giving a complete description when $G$ is nilpotent and proving, among the others, that $\Gamma(S_n)$ and $\Gamma(A_n)$ are perfect if and only if $n\leq 4$). Finally we prove that for a finite group $G$, the properties that $\Gamma(G)$ is split, chordal or $C_4$-free are equivalent.
	\end{abstract}

\maketitle

\section{Introduction}

If a finite group $G$  can be generated by $d$ elements, the question of which sets of $d$ elements of $G$ generate $G$ is nontrivial. The simplest interesting case is when $G$ is 2-generated. One tool developed to study generators of 2-generated finite groups is the generating graph $\Gamma(G)$ of $G.$ This is the graph which has the elements of $G$ as vertices and an edge between two elements $g_1$ and $g_2$ if $G$ is generated by $g_1$ and $g_2.$ Note that the generating graph may be defined for any group, but it only has edges if $G$ is $2$-generated.

\  

Several strong structural results about $\Gamma(G)$ are known in the case where $G$ is simple, and this reflects the rich group theoretic structure of these groups. For example, if $G$ is a nonabelian simple group, then the only isolated vertex of $\Gamma(G)$ is the identity \cite{gk} and the graph $\Delta(G)$ obtained by removing the isolated vertex is connected with diameter two \cite{bgk} and, if $|G|$ is sufficiently large, admits a Hamiltonian cycle \cite{bglm} (it is conjectured that the condition on $|G|$ can be removed). Moreover, in recent years there has been considerable interest in attempting to classify the groups $G$ for which $\Gamma(G)$ shares the strong properties of the generating graphs of simple groups. Recently it has been proved the remarkable result that
the identity is the unique isolated vertex of $\Gamma(G)$ if and only if all proper quotients of $G$ are cyclic \cite{bgh}. 
An open question is whether the subgraph $\Delta(G)$ of $\Gamma(G)$ induced by the non-isolated vertices is connected, for every finite group $G.$
The answer is positive if $G$ is soluble \cite{CL3} and in this case
the diameter of $\Delta(G)$ is at most three \cite{diam}. In \cite{hl} it is proved that when $G$ is nilpotent, then $\Delta(G)$ is maximally connected.

\

A number of important graph classes, can be defined either structurally or
in terms of forbidden induced subgraphs. The aim of this paper is to investigate some properties of the forbidden subgraphs of the generating graph.

\

A perfect graph is a graph in which the chromatic number of every induced
subgraph equals the order of the largest clique of that subgraph (clique number). A hole in a graph $\Gamma$ is an induced subgraph of $\Gamma$ isomorphic to a chordless cycle of length at least 4. An antihole is an induced subgraph $\Delta$ of $\Gamma$, such that $\overline{\Delta}$ is a hole of the complement graph $\overline{\Gamma}$. A hole
(resp. an antihole) is odd or even according to the number of its vertices.  The strong perfect graph theorem is a forbidden graph characterization of the perfect graphs as being exactly the graphs that have neither odd
holes  nor odd antiholes. It was conjectured by Claude Berge in 1961. A proof by Maria Chudnovsky, Neil Robertson, Paul Seymour and Robin Thomas was announced in 2002 and published by them in 2006 \cite{crst}. Motivated by the strong perfect graph theorem we analyze the existence of $m$-holes or $m$-antiholes in the generating graph of a finite group $G.$ The first result that can be proved with this approach is a complete characterization of the 2-generated finite nilpotent groups with a perfect generating graph.

\begin{thm}\label{nilperf}
	Let $G$ be a finite 2-generated nilpotent group. Then $\Gamma(G)$ is perfect if and only if the index of the Frattini subgroup is the product of at most 4 primes.
\end{thm}

In general the condition on the number of prime divisors of the index of the Frattini subgroup is neither necessary nor sufficient to ensure that the generating graph is perfect, as it follows for example from the study of the generating graph of dihedral groups.

\begin{thm}\label{diedrale}Let $D_n$ be the dihedral group of order $2n$ and degree $n.$ Then
	$\rg(D_n)$ is perfect \ifa one of the following occurs:
	\begin{enumerate}
		\item $n$ is even;
		\item $n$ is odd, divisible by at most two distinct primes.	\end{enumerate}
\end{thm}

An interesting and surprising consequence of Theorem \ref{diedrale} is that if $G$ is a 2-generated finite group and $N$ is a normal subgroup of $G$, then the fact that $\Gamma(G)$ is perfect, does not imply that
$\Gamma(G/N)$ is also perfect.  For example let $m=p_1\cdot p_2 \cdot p_3$ be the product of three distinct primes and let $G=D_{2m}$ be the dihedral group of order $4m.$ By Theorem  \ref{diedrale}, $\Gamma(G)$ is perfect, however $G$ has a normal subgroup $N$ of order 2 such that $G/N \cong D_m$ and, again by Theorem \ref{diedrale}, $\Gamma(G/N)$ is not perfect.

We will prove (see Theorem \ref{mina5}) that the alternating group $A_5$ is the smallest 2-generated finite group whose generating graph is not perfect. Moreover:
\begin{thm}\label{altesim}
	$\rg(A_n)$ and $\rg(S_n)$ are perfect \ifa $n < 5$.
\end{thm}

The behaviour of the generating graph of the alternating groups suggest the following conjecture.

\begin{con}\label{cong1}
	If $G$ is a finite non-abelian simple group, then $\Gamma(G)$ is not perfect.
\end{con}

Indeed, the proof of Theorem \ref{altesim} shows that if $n\geq 5$ then $\Gamma(A_n)$ and $\Gamma(S_n)$ contain a 5-hole, so we may also formulate a stronger conjecture.

\begin{con}\label{cong2}
	If $G$ is a finite non-abelian simple group, then there exists a subset $X$ of $G$ such that the subgraph of $\Gamma(G)$ induced by $X$ is a 5-hole.
\end{con}

With the use of GAP \cite{GAP4}, we have checked the existence of a $5$-hole in $\rg(G)$, if $G$ is the Tits group or one of the sporadic simple groups with the exception of the Janko group $J_4$, the Thompson group, the Lyons group, the Baby Monster group and the Monster group.
Moreover  Conjecture \ref{cong2} is true 
when $G\cong \psl_2(q)$, $G\cong {^2B}_2(q)$ or $G\cong {^2G}_2(q)$ (see subsection \ref{semp}).

\

A path graph  is a graph whose vertices can be listed in the order $v_1, v_2, \dots, v_n$ such that the edges are $\{v_i, v_{i+1}\}$ where $i = 1, 2, \dots, n-1.$ A path graph with $n$-vertices is usually denoted by $P_n.$ A graph $\Gamma$ is called a cograph if $\Gamma$  has no induced subgraph isomorphic to the four-vertex path $P_4.$ Several alternative characterizations of cographs can be given:  a cograph is a graph all of whose induced subgraphs have the property that any maximal clique intersects any maximal independent set in a single vertex;
a cograph is a graph in which every nontrivial induced subgraph has at least two vertices with the same neighbourhoods;
a cograph is a graph in which every connected induced subgraph has a disconnected complement; a cograph is a graph all of whose connected induced subgraphs have diameter at most 2. We will prove
that if $N$ is a normal subgroup of a  2-generated finite group $G$ and   $\Gamma(G/N)$ contains an induced
subgraph isomorphic to $P_n$, then so does $\Gamma(G)$ (see Lemma \ref{quozienti}). Thus, unlike the property that the generating graph is perfect, the property that $\Gamma(G)$ is a cograph is inherited by the epimorphic images of $G.$ This is a considerable advantage in the study of groups whose generating graph is a cograph and allows to obtain quite general results. For example we can completely characterize the 2-generating finite soluble group whose generating graph is a cograph.

\begin{thm}\label{gencog}
	Let $G$ be a 2-generated finite soluble group. Then $\Gamma(G)$ is a cograph if and only if one of the following occurs.
	\begin{enumerate}
		\item $G$ is cyclic, and the order of $G$ is divisible by at most two different prime numbers.
		\item $G$ is a $p$-group.
		\item $G/\frat(G) \cong V \rtimes \langle x \rangle$ where $x$ has prime
		order
		and $V$ is a faithful irreducible $\langle x \rangle$-module.
	\end{enumerate}
\end{thm}

Moreover we will prove the following theorems.

\begin{thm}\label{unicoiso}
	Let $G$ be a finite group and assume that the identity element is the unique isolated vertex of $\Gamma(G).$	If  $\Gamma(G)$ is a cograph, then $G$ is soluble.
\end{thm}
\begin{thm}\label{perf}
	Let $G$ be a  2-generated finite group. If $\Gamma(G)$ is a cograph and $N$ is a maximal normal subgroup of $G,$ then $G/N$ is abelian.
\end{thm}

\begin{cor}
	Let $G$ be a non-trivial 2-generated finite group. If $G$ is perfect, then $\Gamma(G)$ is not a cograph.
\end{cor}

The previous result suggests the following stronger conjecture.

\begin{con}\label{cong4}
	Let $G$ be a 2-generated finite group. If $\Gamma(G)$ is a cograph, then $G$ is soluble.
\end{con}

A graph is chordal if it contains no induced cycle of length greater then
3. A graph is called split if its vertex set is the disjoint union of two
subsets $A$ and $B$ so that $A$ induces a complete graph and $B$ induces
an empty graph. In the final part of the paper, we will prove the following result.

\begin{thm}Let $G$ be a 2-generated finite group. Then the following conditions are equivalent.
	\begin{enumerate}
		\item $\Gamma(G)$ is split.
		\item $\Gamma(G)$ is chordal.
		\item $\Gamma(G)$ is $C_4$-free.
		\item Either $G$ is a cyclic $p$-group or $|G|=2p$ for some prime $p.$
\end{enumerate}\end{thm}

\section{Cographs}

Our first result is that if $\Gamma(G)$ is a cograph, then $\Gamma(G/N)$ is also a cograph, for every normal subgroup $N$ of $G.$
In order to prove a more general statement which implies the previous sentence, we need to recall an auxiliary result, which generalizes an argument due to  Gasch\"utz \cite{g2}. Given a subset $X$ of a finite group $G,$ we will denote by $d_X(G)$ the smallest cardinality
of a set of elements of $G$ generating $G$ together with the elements of $X.$ In the particular case when $X=\emptyset$, $d_\emptyset(G)=d(G)$
is the smallest cardinality of a generating set of $G.$

\begin{lemma}{\cite[Lemma 6]{CL3}}\label{modg} Let $X$ be a subset of $G$
	and $N$ a normal subgroup of $G$ and suppose that
	$\langle g_1,\dots,g_r, X, N\rangle=G.$
	If $r\geq d_X(G),$ we can find $n_1,\dots,n_r\in N$ so that $\langle g_1n_1,\dots,g_rn_r,X\rangle=G.$
\end{lemma}

\begin{lemma}\label{quozienti}
	Let $G$ be a 2-generated finite group and $N$ a normal subgroup of $G$ and let $t\in \mathbb N$ with $t\geq 2.$ If $\Gamma(G/N)$ contains an induced
	subgraph isomorphic to $P_t$, then so does $\Gamma(G).$
\end{lemma}
\begin{proof}Assume that $(a_1N,a_2N,\dots,a_tN)$ is a $t$-vertex path in
	$\Gamma(G/N).$ By Lemma \ref{modg} there exist $n_1, n_2 \in N$ such that
	$\langle a_1n_1, a_2n_2 \rangle=G.$ In particular $d_{\{a_2n_2\}}(G)\leq 1,$ so, again by Lemma \ref{modg}, if $t\geq 3$ then there exists $n_3 \in N$ such that $\langle a_2n_2, a_3n_3 \rangle=G.$
	By	repeating this argument, we can find $n_1,\dots,n_t\in N$ such that
	$\langle a_in_i, a_{i+1}n_{i+1}\rangle=G$ for $1\leq i\leq t-1.$ If $(r,s) \neq (i,i+1)$ for some $i\in \{1,\dots,t-1\}$, then $\langle a_r,a_s\rangle N \neq G,$ and consequently $\langle a_rn_r,a_sn_s\rangle\neq G.$
	So $(a_1n_1,\dots,a_tn_t)$ is a $t$-vertex path in $\Gamma(G).$
\end{proof}

\begin{proof}[Proof of Theorem \ref{unicoiso}] 
	This can be proved with the same argument used by Cameron in \cite[Theorem 8.8]{cam}. Let $\Delta(G)$ be the subgraph of $\Gamma(G)$ obtained by deleting the identity element. By \cite[Theorem 1]{bgh} the graph $\Delta(G)$ is connected. The join graph of $G$ is the graph whose vertices are the non-trivial proper subgroups of $G$ and in which two vertices $H$ and $K$ are adjacent if and only if $H\cap K\neq 1.$ By \cite{she} if $G$ is not soluble, then this graph is connected. It can be easily seen that this implies that the complement graph $\overline{\Delta(G)}$ is connected. Since the graph complement of a connected cograph is disconnected, it follows that $\Delta(G)$ (and consequently $\Gamma(G)$) is not a cograph when $G$ is not soluble.
\end{proof}

\begin{proof}[Proof of Theorem \ref{perf}] If follows immediately combining Lemma \ref{quozienti} and Theorem \ref{unicoiso}.
\end{proof}

\begin{lemma}\label{abelian}
	Let $G$ be a 2-generated finite nilpotent group. If $\Gamma(G)$ is a cograph, then $|G/\frat(G)|$ is the product of at most two primes.
\end{lemma}
\begin{proof}
	Assume that $|G/\frat(G)|$ is divisible by $p_1p_2p_3$, with $p_1, p_2, p_3$ prime numbers. Since $d(G)\leq 2,$ it cannot be $p_1=p_2=p_3,$ so we may assume $p_3\notin \{p_1, p_2\}.$ Consider $X=\langle x_1 \rangle \times \langle x_2 \rangle \times
	\langle x_3 \rangle$ with $|x_i|=p_i$ for $1\leq i \leq 3.$ It can be easily checked that $(x_1,1,1), (1, x_2,x_3),$ $(x_1,1,x_3),$ $(1,x_2,1)$ is
	a four-vertex path in $\Gamma(X).$ Since $X$ is an epimorphic image of $G,$  Lemma \ref{quozienti} would imply that $\Gamma(G)$ is not a cograph.
\end{proof}

\begin{lemma}\label{semidiretto}
	Let $H$ be a 2-generated finite soluble group and $V$ a non-trivial irreducible $H$-module. Assume that  there exist $a,b \in H$ such that
	\begin{enumerate}
		\item $H=\langle a, b \rangle$;
		\item $H \neq \langle a\rangle,  H \neq \langle b\rangle;$
		\item $a \notin C_H(V), b \notin C_H(V).$
	\end{enumerate}
	Consider the semidirect product $G=V\rtimes H.$ If no complemented  chief factor of $H$ is $H$-isomorphic to $V$, then $\Gamma(G)$ contains a subgraph isomorphic to the four-vertex path $P_4.$
\end{lemma}

\begin{proof}
	Let $|V|=p^t,$ with $p$ a prime. Define
	$$\Omega_a=\{v\in V\mid \langle a, bv\rangle=G\},  \quad \Omega_b=\{v\in V\mid \langle av, b\rangle=G\}.$$
	Assume $v\notin \Omega_a.$ Then $\langle a, bv\rangle$ is a complement of
	$V$ in $G$. The fact that no  complemented  chief factor of $H$ is $H$-isomorphic to $V$ ensures that all the complements of $V$ in $G$ are conjugated (see \cite[Satz 3]{g1}), so there exists $w\in V$ such that
	$(a,bv)=(a^w,b^w).$ In particular $w\in C_V(a)$ and $v=[b,w].$ This implies $|V\setminus \Omega_a|\leq |[b,C_V(a)]|\leq |C_V(a)|.$ Since we are assuming $C_V(a)<V$, we deduce
	\begin{equation}
	\label{omega}|\Omega_a|\geq |V|-|C_V(a)|\geq p^t-p^{t-1}.
	\end{equation}
	For the same reason
	\begin{equation}
	\label{omegb}|\Omega_b|\geq |V|-|C_V(b)|\geq p^t-p^{t-1}.
	\end{equation}
	Let $\Omega=\{(v_1,v_2) \in V^2\mid \langle av_1,bv_2\rangle=G\}.$
	The number of pairs $(v_1,v_2)$ in $V^2\setminus \Omega$ coincides with the number of complements of $V$ in $G$, so
	\begin{equation}\label{omeg}
	|\Omega|=|V^2|-|V|.
	\end{equation}
	If $(v_1,v_2)\in \Omega \cap (\Omega_b\times \Omega_a)$ then
	$(a,bv_2,av_1,b)$ is a four-vertex path in $\Gamma(G).$
	In particular, if $|\Omega_a \times \Omega_b|+|\Omega|>|V|^2,$ then $(\Omega_a \times \Omega_b) \cap \Omega \neq \emptyset,$ and $G$ contains $P_4$. So we may assume
	\begin{equation}\label{cru}|\Omega_a||\Omega_b|\leq |V^2| - |\Omega|=|V|.
	\end{equation}
	In particular it follows from (\ref{omega}),  (\ref{omegb}) and  (\ref{omeg}), that $(p^t-p^{t-1})^2\leq p^t,$ i.e.
	\begin{equation}\label{s3}
	p^t\leq \left(\frac{p}{p-1}\right)^2.
	\end{equation}
	This implies
	$p=2$ and $t=2$, i.e. $V\cong C_2\times C_2$. We have two possibilities:
	
	\noindent a) $H/C_H(V) \cong \gl(2,2) \cong \perm(3)$.
	In this case $G/C_H(V)\cong \perm(4).$ Since 
	$((1,2), (2,3,4), (1,4), (1,2,3))$ is a four-vertex path in $\perm(4),$ the conclusion follows from
	Lemma \ref{quozienti}.
	
	\noindent b) $H/C_H(V) \cong C_3.$ In this case $C_V(a)=C_V(b)=\{0\}$, but then, by (\ref{omega}) and  (\ref{omegb}), $|\Omega_a|,|\Omega_b|\geq 3$, in contradiction with (\ref{cru}).
\end{proof}

\begin{lemma}\label{nonnil}
	Let $G$ be a non-nilpotent 2-generated finite soluble group. If $\Gamma(G)$ is a cograph, then $G/\frat(G) \cong N \rtimes H,$ where $N$ is a faithful irreducible $H$-module and $H$ is cyclic of prime  order.
\end{lemma}
\begin{proof}
	Assume that $\Gamma(G)$ is a cograph. Then also $\Gamma(G/\frat(G))$ is a cograph. Moreover $G/\frat(G)$ is not nilpotent (otherwise $G$ would be nilpotent) so it is not restrictive to assume $\frat(G)=1.$ Since $G$ is
	not nilpotent, there exists a minimal normal subgroup of $G$, say $N$,  which is not central in $G$. Set $H=G/C_G(N).$  Then $N$ is a faithful irreducible $H$-module and the semidirect product $N\rtimes H$ is an epimorphic image of $G.$ By Lemma \ref{quozienti}, $\Gamma(N \rtimes H)$ is a cograph, so it follows from Lemma \ref{semidiretto} that $H$ is a cyclic group and consequently $\dim_{\End_H(N)}N=1$. Let $K$ be a complement of
	$N$ in $G.$ Since $G$ is 2-generated and  $\dim_{\End_H(N)}N=1$, it follows from \cite[Satz 4]{g2} that no complemented  chief factor of $K$ is $K$-isomorphic to $N$. Since $G/C_G(N)$ is cyclic, there exists $x \in K$
	such that $K=\langle x, C_K(N)\rangle.$ Moreover, since $K$ is 2-generated, by   Lemma \ref{modg} there exist $c_1,c_2 \in C_K(N)$ such that
	$\langle xc_1, xc_2\rangle=K$. If $K$ is not cyclic, then the two elements $a=xc_1$ and $b=xc_2$ satisfy the assumptions of Lemma
	\ref{semidiretto}. But this would imply that $G$ is not a cograph, a contradiction.
	With a similar argument we can prove that $K/C_K(N)$ is a $p$-group. Indeed assume $|K/C_K(N)|=rs$ with $r,s \geq 2$ and $(r,s)=1.$ There exist $y_1, y_2 \in K$ such that $\langle y_1, y_2\rangle=K$, $|y_1C_K(N)|=r$ and $|y_2C_K(N)|=s.$ We take $y_1, y_2$ in the role of $a, b$ in Lemma \ref{semidiretto} and we deduce that $G$ is not a cograph. 
	So we may assume $K=\langle xy \rangle$ where $|x|$ is a $p$-power,
	$y \in C_K(N)$ and $(|y|,p)=1.$ If $y\neq 1,$ then, for any
	$1\neq n \in N,$ $(n,xy,ny,x)$ is a four-vertex path in $\Gamma(G).$ So $y=1$ and $K$ is a cyclic $p$-group. In particular $C_K(N)\leq \frat(K).$ However $\frat(K)\cap C_K(N)\leq \frat(G)=1,$ so we deduce that $C_K(N)=1.$ We have so proved that $K=\langle x \rangle$ is cyclic of order $p^t$, for some $t \in \mathbb N$ and $N$ is a faithful irreducible $K$-module. In particular $K$ acts fixed-point-freely on $N$. Choose $1\neq n\in N.$ Then $K$ and $K^n$ are two maximal subgroups of $G$ with trivial
	intersection. If $t>1,$ then
	$(x^p,x^n,x,(x^n)^p)$ is a four-vertex path in $\Gamma(G)$. Since $\Gamma(G)$ is a cograph we conclude $t=1.$
\end{proof}

\begin{proof}[Proof of Theorem \ref{gencog}] Assume that $\Gamma(G)$ is a cograph. If $G$ is nilpotent then, by Lemma \ref{abelian}, $G/\frat(G)$ is either a $p$-group or a cyclic group of order $p_1p_2$, where $p_1$ and $p_2$ are two different primes. In the first case $G$ is a $p$-group, in the second $G$ is a cyclic group and $p_1, p_2$ are the only prime divisors of $|G|.$ If $G$ is not nilpotent, then, by Lemma \ref{nonnil},
	$G/\frat(G) \cong V \rtimes \langle x \rangle$ where $x$ has prime order
	and $V$ is a faithful irreducible $\langle x \rangle$-module.
	
	Conversely we have to prove that if $G$ satisfies (1), (2) or (3), then $\Gamma(G)$ is a cograph. If $(g_1,g_2,g_3,g_4)$ is a four-vertex path in $\Gamma(G)$, then either $(g_1\frat(G),$ $g_2\frat(G), g_3\frat(G), g_4\frat(G))$ is a four-vertex path in $\Gamma(G/\frat(G))$ or there exist $1\leq i < j \leq 4$ with $g_i\frat(G)=g_j\frat(G).$ However the second possibility cannot occur, indeed there exists $k \in \{1,2,3,4\} \setminus \{i,j\}$ such that
	$g_k$ is adjacent to $g_j$ but not to $g_i.$ This implies $G=\langle g_k, g_j\rangle=\langle g_k, g_j, \frat(G) \rangle  = \langle g_k, g_i, \frat(G) \rangle < G,$ a contradiction. Therefore it suffices to prove that in our hypotheses and under the additional assumption  $\frat(G)=1,$ the generating graph $\Gamma(G)$ is a cograph. Assume by contradiction that $(g_1,g_2,g_3,g_4)$ is a  four-vertex path in $\Gamma(G)$. We have the following possibilities.
	
	\noindent a) $G \cong C_p.$  There exists $i\in \{1,2,3,4\}$ such that
	$|g_i|=p,$ but then $g_i$ is adjacent to $g_j$ for any $j\neq i,$ a contradiction.
	
	\noindent b) $G \cong C_p \times C_p.$  In this case $|g_1|=|g_2|=|g_3|=|g_4|=p.$ Since $g_1$ and $g_3$ are not adjacent in $\Gamma(G),$ $\langle g_1 \rangle=\langle g_3 \rangle.$ Moreover, since $g_3$ and $g_4$ are  adjacent $\Gamma(G),$ $\langle g_3 \rangle\neq \langle g_4 \rangle.$ But then $\langle g_1 \rangle \neq \langle g_4 \rangle$ and $g_1$ and $g_4$ are adjacent in $\Gamma(G),$ a contradiction.
	
	\noindent c) $G \cong C_{p_1} \times C_{p_2},$ with $p_1\neq p_2.$ 
	There is no  $i\in \{1,2,3,4\}$ such that
	$|g_i|=p_1p_2,$ since this would imply  $g_i$  adjacent to $g_j$ for any $j\neq i.$ It is not restrictive to assume $|g_1|=p_1.$ This would imply $|g_2|=p_2, |g_3|=p_1, |g_4|=p_2$, and consequently that $g_1$ and $g_4$ are adjacent, a contradiction.
	
	\noindent d) $G/\frat(G) \cong V \rtimes \langle x \rangle$ where $x$ has
	order $p$
	and $V$ is a faithful irreducible $\langle x \rangle$-module. 
	There exists a prime $q\neq p$ such that $V$ is an elementary abelian $q$-group and a non-trivial element $g$ of $G$ has either order $p$ or order
	$q.$ Assume that $|g_1|=p.$ Then $\langle g_1 \rangle$ is the unique maximal subgroup of $G$ containing $g_1.$ Since $g_3$ and $g_4$ are not adjacent to $g_1,$ we must have $g_3,g_4 \in \langle g_1 \rangle$, but then $g_3,g_4$ are not adjacent in $\Gamma(G).$ So $|g_1|=q.$ For the same reason $|g_4|=q$ and consequently $|g_2|=|g_3|=p.$ But this would imply that $g_2$ and $g_4$ are adjacent.
\end{proof}

\section{Perfect graphs}

In the following, we will denote with $Y$ the following graph:
\begin{center}
	\begin{tikzpicture}
	[scale=.6,auto=left,every node/.style={circle,fill=black!20}]
	\node (n1) at (5,3.5)  {$x_2$};
	\node (n2) at (5,6.5) {$x_1$};
	\node (n3) at (7,5)  {$x_3$};
	\node (n5) at (9,5) {$x_4$};
	\foreach \from/\to in {n1/n2,n1/n3,n2/n3,n5/n3}
	\draw (\from) -- (\to);
	\end{tikzpicture}
\end{center}

\begin{thm}\cite[Theorem 3.2]{rav}\label{rv}
	The tensor product $\Gamma_1 \wedge \Gamma_2$ is perfect if and only if  either
	\begin{itemize}
		\item $\Gamma_1$ or $\Gamma_2$ is bipartite, or
		\item both $\Gamma_1$ or $\Gamma_2$ do not contain and odd $n$-hole with
		$n\geq 5,$ or $Y$ as an induced subgraph.
	\end{itemize}
\end{thm}
\begin{rem}\label{yk3}Let $\Gamma_1\cong Y$ be a graph with vertex-set $\{x_1,x_2,x_3,x_4\}$ and $\Gamma_2 \cong K_3$ be a complete graph with
	vertex-set $\{y_1,y_2,y_3\}.$ Then $$((x_1,y_1), (x_2,y_3), (x_3,y_1), (x_4,y_2), (x_3,y_3))$$ is a 5-hole in the tensor product $\Gamma_1 \wedge \Gamma_2.$
\end{rem}
\begin{lemma}\label{qfra} Let $\frat(G)$ be the Frattini subgroup of a finite group $G.$ 
	Then  $\Gamma(G)$ is perfect if and only if $\Gamma(G/\frat(G))$ is perfect.
\end{lemma}

\begin{proof}
	By the strong graph perfect theorem, it suffices to prove that, when $m\geq 5,$
	$\Gamma(G)$ contains an $m$-hole or an $m$-antihole if and only if so does $\Gamma(G/\frat(G)).$	Since $\langle g_1\frat(G), g_2\frat(G)\rangle=G/\frat(G)$ if and only if $\langle g_1, g_2 \rangle=G,$ if the subset $\{x_1\frat(G),\dots,x_m\frat(G)\}$ induces an $m$-hole or an $m$-antihole in $\Gamma(G/\frat(G))$, then so does
	$\{x_1,\dots,x_m\}$ in $\Gamma(G).$ Conversely, assume that
	$\{x_1,\dots,x_m\}$ induces an $m$-hole or an $m$-antihole in $\Gamma(G).$ If $1\leq i< j\leq m,$ then there exists $k\in \{1,\dots,m\}\setminus \{i,j\}$ such that $x_k$ is adjacent to $x_i$ but not to $x_j$, in particular $x_i\frat(G)\neq x_j \frat(G)$ and $\{x_1\frat(G),\dots,x_m\frat(G)\}$ induces an $m$-hole or an $m$-antihole in $\Gamma(G/\frat(G)).$
\end{proof}

Let $I_n=\{1,\dots,n\}$ and consider the graph $\Delta_n$ whose vertices are the subsets of $I_n$ and where $J_1$ and $J_2$ are adjacent if and only if $J_1\cup J_2=I_n.$

\begin{lemma}\label{deltan}
	The graph $\Delta_n$ is perfect if and only if $n\leq 4.$
\end{lemma}
\begin{proof}
	If $n\geq 5,$ then $
	(\{1,2,4,6,\dots,n\}, \{1,3,5,6,\dots,n\},$ $\{2,4,5,6,\dots,n\},$\linebreak $\{1,3,4,6,\dots,n\},$  $\{2,3,5,6,\dots,n\})$ is a 5-hole in $\Delta_n$ so $\Delta_n$ is not perfect. We may assume $n\leq 4.$ Let $m\geq 5$ be an odd integer
	and assume that $X$ is a subset of the vertex-set of $\Delta_n$ inducing an  $m$-hole or an $m$-antihole. Clearly $I_n \notin  X.$ As a consequence, $\emptyset\notin X.$ Moreover if $\{i\}$ is a singleton, then $I_n\setminus\{i\}$ is the unique proper subset of $I_n$ adjacent to $\{i\}$, so $\{i\}\notin X$. So we have at most $2^n-n-2$ possible choices for an element of $X$. This implies that $n=4$ and $X$ consists of sets of cardinality 2 or 3. Since $I_4$ contains only four subsets of cardinality 3, it is not restrictive to assume $(1,2)\in X.$ Note that
	a subset of cardinality 3 is adjacent to all the other subsets of cardinality 3. So if $X$ induces an $m$-hole, then $X$ contains at most 2 (adjacent) subsets of cardinality 3. This implies that $X$ contains at least 3 subsets of cardinality 2, inducing a 3-vertex path. But this is impossible since  a subset of cardinality 2 is adjacent to only one subset of cardinality 2. If $X$ induces an $m$-antihole, then it contains at least one subset of cardinality 2,
	say $Y$, and this must be adjacent to other $m-3$ elements of $X$. However
	there are is a unique subset of cardinality  2 and  two subsets of cardinality 3 adjacent to $Y$, hence $m-3\leq 3.$
	But this implies $m=5$ and we may exclude this possibility since a $5$-antihole is isomorphic to a $5$-hole.
\end{proof}

\begin{lemma}\label{unicomax}
	Let $g\in G$ be an element which is contained in a unique maximal subgroup of $G$.  Then $g$ cannot be the vertex of an $m$-hole or $m$-antihole in $\rg(G)$ with $m\geq 5$.
\end{lemma}
\begin{proof}Let $M\leq G$ be the unique maximal subgroup containing $g$.
	\begin{itemize}
		\item Let $(g,a_2,\dots,a_m)$ be an $m$-hole. We have $g\nsim a_3,a_4$, which implies $a_3,a_4\in M$, so they cannot be adjacent in $\Gamma(G),$ a contradiction.
		\item Let $(g,a_2,\dots,a_m)$ be an $m$-antihole. We have $g\nsim a_2,a_m$, which implies $a_2,a_m\in M$, so they cannot be adjacent in $\Gamma(G),$ a contradiction.\qedhere
	\end{itemize}
\end{proof}

\begin{lemma}\label{aiaj}
	Let $m\geq5$ and suppose $(a_1,\dots,a_m)$ is an $m$-hole or an $m$-antihole in $\rg(G)$. If $\gen{a_i}=\gen{a_j}$, then $i=j$.
\end{lemma}
\begin{proof}
	Let $i\neq j$ and $\gen{a_i}=\gen{a_j}$. We can assume without loss of
	generality that $i=1$ and $2\leq j\leq \frac{m+1}{2}$. If $(a_1,\dots,a_m)$ is an $m$-hole, then $a_m\sim a_1$, and this implies $a_m\sim a_j$ and consequently $j=m-1$. But then $m-1\leq\frac{m+1}{2}$, hence $m\leq3$, a contradiction. If $(a_1,\dots,a_m)$ is an $m$-antihole, then $a_m\nsim
	a_1$, and so $a_m\nsim a_j$ and we argue as before.
\end{proof}

\subsection{Nilpotent groups}
The aim of this subsection is to prove Theorem \ref{nilperf}.
First we prove the statement in the particular case when $G$ is cyclic.
\begin{lemma}\label{ciclo}
	Let $G$ be a finite cyclic group. Then $\Gamma(G)$ is perfect if and only
	if $|G|$ is divisible by at most four different primes.
\end{lemma}
\begin{proof}
	By Lemma \ref{qfra}, we may assume $\frat(G)=1$, so $|G|=p_1\cdots p_t$ where $p_1,\dots,p_t$ are distinct primes. Assume that $(a_1,\dots,a_m)$ is an $m$-hole or an $m$-antihole in $\Gamma(G).$ Let $\pi=\{p_1,\dots,p_t\}$ and for any $i\in \{1,\dots,t\}$, let $\pi_i$ be the set of prime divisors of $|a_i|$. By Lemma \ref{aiaj}, if $i\neq j$, then $\pi_i \neq \pi_j$, moreover
	$a_i$ and $a_j$ are adjacent in $\Gamma(G)$ if and only if $\pi_i \cup \pi_j=\pi.$ This implies that $\Gamma(G)$ is perfect if and only if $\Delta_t$ is perfect, and the conclusion follows from Lemma \ref{deltan}.
\end{proof}

The proof of the general case requires some preliminary lemmas and remarks.

\begin{rem}\label{y1}
	Let $p$ and $q$ be two different primes. If $P=\langle a_1, a_2 \rangle$ is a finite, 2-generated, non-cyclic $p$-group and $Q=\langle b_1, b_2 \rangle$ is a finite, 2-generated, non-cyclic $q$-group, then $\Gamma(P\times Q)$ contains an induced subgroup isomorphic to $Y:$
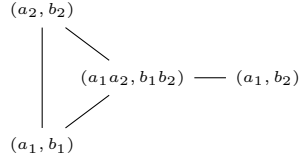
	\begin{center}
		\begin{tikzpicture}
		[scale=.6,auto=left]
		\node (n1) at (5,3.5)  {\tiny{$(a_1,b_1)$}};
		\node (n2) at (5,6.5) {\tiny{$(a_2,b_2)$}};
		\node (n3) at (7,5)  {\tiny{$(a_1a_2,b_1b_2)$}};
		\node (n5) at (10,5) {\tiny{$(a_1,b_2)$}};
		\foreach \from/\to in {n1/n2,n1/n3,n2/n3,n5/n3}
		\draw (\from) -- (\to);
		\end{tikzpicture}
	\end{center}
\end{rem}

\begin{rem}\label{y2}
	If $P=\langle a_1, a_2 \rangle$ is a finite, 2-generated, non-cyclic finite $p$-group and $C=\langle x \rangle$ is a non-trivial finite cyclic group whose order is not divisible by $p$, then $\Gamma(P\times C)$ contains an induced subgroup isomorphic to $Y:$
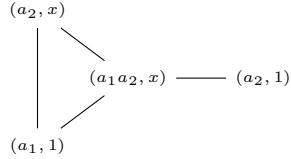
	\begin{center}
		\begin{tikzpicture}
		[scale=.6,auto=left]
		\node (n1) at (5,3.5)  {\tiny{$(a_1,1)$}};
		\node (n2) at (5,6.5) {\tiny{$(a_2,x)$}};
		\node (n3) at (7,5)  {\tiny{$(a_1a_2,x)$}};
		\node (n5) at (10,5) {\tiny{$(a_2,1)$}};
		\foreach \from/\to in {n1/n2,n1/n3,n2/n3,n5/n3}
		\draw (\from) -- (\to);
		\end{tikzpicture}
	\end{center}
\end{rem}

\begin{rem}\label{k3}
	If $G$ is a 2-generated finite group of order at least 3, then $\Gamma(G)$ contains an induced subgraph isomorphic to $K_3.$ In particular $\Gamma(G)$ is not a bipartite graph.
\end{rem}
\begin{proof}
	If $G=\langle a, b\rangle$ is not cyclic, then we can take the subgraph
	of $\Gamma(G)$ induced by $\{a,b,ab\}.$ If $G=\langle x\rangle,$ we can
	take the subgroup induced by $\{1,x,x^{-1}\}.$
\end{proof}

\begin{lemma}\label{necc}
	Let $G$ be a 2-generated finite nilpotent group. If $\Gamma(G)$ is perfect, then the order of $G/\frat(G)$ is the product of at most four (not necessarily distinct) primes.
\end{lemma}
\begin{proof}
	By Lemma \ref{qfra} we may assume $\frat(G)=1.$ For any prime divisor $p$ of $|G/\frat(G)|$, the Sylow $p$-subgroup of $G$ is either cyclic of order $p$ or elementary abelian of order $p^2.$ If all the Sylow subgroups of $G$ are cyclic, then $G$ is cyclic and the conclusion follows from Lemma \ref{ciclo}. So we may assume that $G$ contains a non-cyclic Sylow $p$-subgroup, say $P$, of order $p^2.$ Let $K$ be a complement of $P$ in $G.$ Assume, by contradiction, that
	$|K|$ is the product of at least three primes. If $K$ is not cyclic, then $K=Q_1\times Q_2 \times H$ where $Q_1, Q_2$ are Sylow subgroups, $Q_1$ is non-cyclic and $Q_2\neq 1.$ 
	By Remarks \ref{k3} and \ref{y2}, $\Gamma(P\times Q_2)$ and  $\Gamma(Q_1\times H)$ contain an induced subgraph isomorphic, respectively,  
	to $Y$ and $K_3.$ But then we deduce from Remark \ref{yk3} that
	$\Gamma(G) \cong \Gamma(P\times Q_2) \wedge \Gamma(Q_1\times H)$ is not perfect.
	So we may assume that $K=\langle x \rangle$ and that $|x|$ is divisible by at least three different primes $q_1, q_2, q_3.$
	Let $\Omega$ be the set of the vertices $y$ of $\Gamma(K)$ with the property that $\langle y \rangle \neq K$ and let $\Lambda$ be the subgraph of $\Gamma(K)$ induced by $\Omega.$ Notice that the subgroup of $\Gamma(G)$ induced by the subset $P \times \Omega$ is isomorphic with $\Gamma(P) \wedge \Lambda$ and that $\{ x^{q_1},x^{q_2},x^{q_3},x^{q_1q_2}\}$ induces a subgraph of $\Lambda$ isomorphic to $Y$. But then, again by Remark \ref{yk3}, 
	$\Gamma(P) \wedge \Lambda$, and consequently $\Gamma(G)$, contains a 5-hole.
\end{proof}

\begin{lemma}\label{pp}
	Let $G$ be a non-cyclic 2-generated finite $p$-group. Then $\Gamma(G)$ is perfect and does not contain an induced subgroup isomorphic to $Y$.
\end{lemma}	
\begin{proof}	
	We have $G/\frat(G)\cong C_p\times C_p.$	 If $g$ is a non-isolated vertex of $\Gamma(G)$, then $|g\frat(G)|=p$ and $\langle g \rangle$ is the unique maximal subgroup of $G$
	containing $g.$ It follows from Lemma \ref{unicomax} that $\Gamma(G)$ contains no $m$-hole or $m$-antihole with $m\geq 5$, so it follows from the strong perfect graph theorem that $\Gamma(G)$ is perfect. Now assume by contradiction that $\{g_1,g_2,g_3,g_4\}$	induces a subgraph of $\Gamma(G)$
	isomorphic to $Y.$ We may order these four vertices in such a way that $g_1$ and $g_2$ are adjacent while  $g_4$ is not adjacent neither to $g_1$ nor to $g_2.$  The latter condition implies  $\langle g_4\frat(G) \rangle = \langle g_1\frat(G) \rangle
	= \langle g_2\frat(G) \rangle,$ in contradiction with $\langle g_1, g_2\rangle=G.$
\end{proof}

\begin{proof}[Proof of Theorem \ref{nilperf}]
	By Lemma \ref{qfra} we may assume that $\frat(G)=1.$ By Lemma \ref{necc}, the condition that $|G|$ is the product of at most 4 primes is necessary for $\Gamma(G)$ being perfect. We have to prove that this condition is also sufficient. By Lemma \ref{ciclo}, we may assume that $G$ is not cyclic. This means that $G=P \times K,$
	where $P \cong C_p\times C_p$ for a suitable prime $p$ and $K$ is a nilpotent group whose order is coprime with $p$ and is the product of at most
	two primes. By Lemma \ref{pp}, we may assume $K\neq 1.$
	
	If $K$ is not cyclic, then $K\cong C_q \times C_q,$ for a prime $q\neq p$ and $\Gamma(G)\cong \Gamma(P) \wedge \Gamma(Q)$ is perfect, as a consequence of Theorem \ref{rv} and Lemma \ref{pp}.
	
	The previous argument does not work if $K=\langle g\rangle$ is cyclic.
	Indeed it is no more true that $\Gamma(G)\cong \Gamma(P) \wedge \Gamma(K)$. For example, if $P=\langle a_1, a_2\rangle$, then
	$(a_1,g)$ and $(a_2,g)$ are adjacent in $\Gamma(G)$ but not in $\Gamma(P) \wedge \Gamma(K)$. However we can argue in the following way. Assume that $X\subseteq G$ induces an $m$-hole or an $m$-antihole, with $m\geq 5.$ If $K=\langle g \rangle$ and $y \in P,$ then either $(y,g)$ is an isolated vertex of $\Gamma(G)$ (when $y=1$), or $\langle (y,g)\rangle$ is the unique maximal subgroup of $G$ containing $(y,g).$  In both the cases, by the fact that the vertices of an hole or an antihole are not isolated and by Lemma \ref{unicomax},
	it follows that $(y,g) \notin X.$ In particular this excludes that $K$ has prime order (no element of $K$ could belong to $X$), so we remain with
	the case when $K$ is cyclic of order $r\cdot s$, where
	$r$ and $s$ are different primes. In this case consider the subgraph $\Delta$ of $\Gamma(K)$ induced by the elements of $K$ of prime order. From what we said above, it follows that $X$ induces an $m$-hole or $m$-antihole in $\Gamma(G)$ if and only if it induces an $m$-hole or $m$-antihole in $\Gamma(P) \wedge \Delta.$ This would imply that  $\Gamma(P) \wedge \Delta$ is not perfect, and consequently, by
	Theorem \ref{rv} and Lemma \ref{ciclo}, that $\Delta$ contains an induced subgraph isomorphic to $Y.$  So assume by contradiction that $\{g_1,g_2,g_3,g_4\}$	induces a subgraph of $\Delta$ isomorphic to $Y.$ We may order these four vertices in such a way that $g_1$ and $g_2$ are adjacent while  $g_4$ is not adjacent neither to $g_1$ nor to $g_2.$  The latter condition implies  $\langle g_4 \rangle = \langle g_1 \rangle
	= \langle g_2 \rangle,$ in contradiction with $\langle g_1, g_2\rangle=K.$
\end{proof}

\subsection{The dihedral group}
In this subsection we analyse when the dihedral group
\[
D_n=\gen{\rho,\iota\mid \rho^n=\iota^2=1, \rho^\iota=\rho^{-1}}
\]
has a perfect generating graph. We start with a preliminary lemma.

\begin{lemma}\label{c2c2}
	Let $G$ be a finite group and $N\nn G$ \st $G/N\cong C_2\times C_2$. Then $\rg(G)$ has no $m$-antihole, for $m\geq7$.
\end{lemma}
\begin{proof}
	Let $a,b,c\in G$ be \st $G/N:=\{N,aN,bN,cN\}$. Suppose $(a_1,\dots,a_m)$ is a $m$-antihole. Since $a_1$ and $a_3$ are adjacent vertices of $\Gamma(G),$ we may assume without loss of generality that $a_1N=aN$ and $a_3N=bN$. Since $a_5,\dots,a_{m-1}$ are adjacent to both $a_1$ and $a_3$, it follows that $a_5N,\dots,a_{m-1}N$ are all equal to $cN$. In particular, if $m>7$, then $a_5N=a_7N$ implies $a_5\nsim a_7$, a contraction. So we may assume $m=7$. Since $a_4\sim a_1$, $a_4\sim a_6$, $a_1N=aN$
	and $a_6N=cN$, we must have  $a_4N=bN$. Analogously, from $a_2\sim a_4$ and $a_2\sim a_6$ it follows $a_2N=aN$. But now consider $a_7N:$  $a_7N\neq aN$ since $a_7\sim a_2$, $a_7N\neq bN$
	since $a_7\sim a_3$ and  $a_7N\neq cN$ since $a_7\sim a_5$. This would imply $a_7\in N$, and consequently that $a_7$ is an isolated vertex of $\Gamma(G),$ a contradiction.
\end{proof}

\begin{proof}[Proof of Theorem \ref{diedrale}] Let $m\geq5$ be odd. We start with two general remarks.
	
	\noindent $(*)$ No rotation $\rho^i$ can appear in an $m$-hole or $m$-antihole. Indeed if $|\rho^i|<n$, then $\rho^i$ is an isolated vertex of 	 $\rg(D_n)$. If $|\rho^i|=n$, then $\langle \rho\rangle$ is the unique maximal subgroup of $D_n$ containing $\rho^i$ and we conclude using Lemma \ref{unicomax}. So every $m$-hole or $m$-antihole in $\rg(D_n)$ will be assumed of the form $(a_1,\dots,a_m)$ with $a_i=\rho^{x_i}\iota$ for some $x_i\in \mathbb Z.$
	
	\noindent $(**)$	$\gen{\rho^a\iota,\rho^b\iota}=D_n$ \ifa $(a-b,n)=1$. 
	
	First we prove that if $n$ is odd, then (2) is a necessary condition for
	$\Gamma(D_n)$ being perfect.	Suppose $n$ is odd and $n=p^aq^br^ck$ with
	$p,q,r$ distinct primes and $k$ not divisible by them, possibly $k=1$. Consider the elements $\alpha_1,\dots,\alpha_4$, obtained solving the following systems (existence of solutions is guaranteed by the Chinese Reminder Theorem):
	\[
	\begin{cases}
	\alpha_1\equiv 1 &\mod p \\
	\alpha_1\equiv b &\mod q\\
	\alpha_1\equiv -1 &\mod r\\
	(\alpha_1\equiv 1 &\mod k)
	\end{cases}\qquad
	\begin{cases}
	\alpha_2\equiv -1 &\mod p \\
	\alpha_2\equiv -1-b &\mod q\\
	\alpha_2\equiv c &\mod r\\
	(\alpha_2\equiv 1 &\mod k)
	\end{cases}
	\]
	\[
	\begin{cases}
	\alpha_3\equiv 1 &\mod p \\
	\alpha_3\equiv 1 &\mod q\\
	\alpha_3\equiv d &\mod r\\
	(\alpha_3\equiv 1 &\mod k)
	\end{cases}\qquad
	\begin{cases}
	\alpha_4\equiv a &\mod p \\
	\alpha_4\equiv -1 &\mod q\\
	\alpha_4\equiv -c-d &\mod r\\
	(\alpha_4\equiv 1 &\mod k)
	\end{cases}
	\]
	where the conditions in the round brackets are considered only when $k\neq1$, and $a,b,c,d$ are such that
	$$
	a\not\equiv0,-1 \!\! \mod p,\quad b\not\equiv0,-1\!\! \mod q, \quad c,d\not\equiv0\!\! \mod r,\quad  c+d\not\equiv 0\!\! \mod r.
	$$
	It can be easily checked that
	$
	(\iota,\rho^{\alpha_1}\iota,\rho^{\alpha_1+\alpha_2}\iota,\rho^{\alpha_1+\alpha_2+\alpha_3}\iota,\rho^{\alpha_1+\alpha_2+\alpha_3+\alpha_4}\iota)
	$
	is a $5$-hole in $\rg(D_n)$.

	Now we prove that if (1) and (2) are satisfied, then $\Gamma(D_n)$ is perfect. We distinguish the different possibilities.
	
	\noindent a) $n$ is even. Since $D_n$ has an epimorphic image isomorphic to $C_2\times C_2$, by Lemma \ref{c2c2} the graph $\rg(D_n)$ has no $m$-antihole with $m\geq7$. Suppose that $\Gamma(D_n)$ contains an $m$-hole $(a_1,\dots,a_m)$, as described in $(*).$ Since $D_n=\gen{\rho^{x_i}\iota,\rho^{x_{i+1}}\iota}$ for every $i$ (where $m+1$ is considered to be $1$), by $(**)$ we should have $x_{i+1}-x_i$ odd for every $1\leq i\leq m$. Then, consider
	\[
	0=\sum_{i=1}^m x_{i+1}-x_i.
	\]
	The right hand side should be odd, because it is a sum of an odd number
	of odd terms, contradiction. We cannot have $m$-holes and $m$-antiholes (a $5$-hole is the same as a $5$-antihole), so $\rg(D_n)$ is perfect.
	
	\noindent b)  $n=p^a$ for a certain prime $p$.  Suppose that $\Gamma(D_n)$ contains an $m$-hole $(a_1,\dots,a_m)$ in
	$\rg(D_n)$. Since $a_1\nsim a_3,a_4$, we should have that $x_4-x_1$ and $x_3-x_1$ are divisible by $p$, hence their difference (i.e. $x_4-x_3$) should be divisible by $p$ and therefore $a_3\nsim a_4$, a contradiction. Suppose now there is a $m$-antihole  $(a_1,\dots,a_m)$. Since $a_2\nsim a_1,a_3$, the prime $p$ should divide $x_3-x_2$ and $x_2-x_1$ and so $p$ should divide their sum (i.e. $x_3-x_1$), which means $a_1\nsim a_3$, a contradiction.
	
	\noindent c)  $n=p^aq^b$ with $p\neq q$ primes. Suppose there is an $m$-hole $(a_1,\dots,a_m)$ in $\rg(D_n)$. Since $a_1\nsim a_3,a_4$, the differences $x_3-x_1$ and $x_4-x_1$ are divisible by at least one of $p$ or $q$. We may assume
	without loss of generality that $x_3-x_1$ is divisible by $p$. Then  $x_4-x_1$ is divisible by $q$, otherwise $a_3\nsim a_4$. For an analogous reasoning $x_4-x_2$ is divisible by $p$ and $x_m-x_3$ is divisible by $q$. From the fact the $p$ divides $x_4-x_2$, arguing as before we deduce that if $5\leq i \leq m$, then   $x_i-x_2$ is divisible by $q$ when $i$ is odd and by $p$ when $i$ is even. In particular $x_m-x_2$ is
	divisible by $q$ and since $q$ divides also $x_m-x_3$, we have $a_2\nsim a_3$, a contradiction.
	Suppose now there is an $m$-antihole $(a_1,\dots,a_m)$ in $\rg(D_n)$. Since $a_i\nsim a_{i+1}$, the difference $x_{i+1}-x_i$ is divisible by at least one of $p$ or $q$. We have an odd number of possible $i$, so there must be a $k$ \st $x_{k+1}-x_k$ and $x_k-x_{k-1}$ are both divisible by the same prime, which means that $x_{k+1}-x_{k-1}$ is also divisible by the
	same prime, hence $a_{k+1}\nsim a_{k-1}$, a contradiction.
\end{proof}

\subsection{Groups of order $p^aq^b$ and $pqr$}\label{piccoli}

We have seen in the previous subsections that if $G$ is a dihedral group or a 2-generated nilpotent group and $|G|$ is divisible by at most three distinct primes, then $\Gamma(G)$ is perfect. However there exist 2-generated finite groups whose generating graph is not perfect, although their order is divisible only by two distinct primes.
Let $H = C_2^2$ and let $h_1$, $h_2$, $h_3$ be the nontrivial elements
of $H$. Let $p$ be an odd prime number and consider $N = \langle x_1,x_2,x_3\rangle \cong C_{p}^3.$ We may define an action of $H$ on $N$ by setting
$$	\begin{aligned}
x_1^{h_1}=x_1, \quad x_1^{h_2}=x_1^{-1}, \quad x_1^{h_3}=x_1^{-1},\\
x_2^{h_1}=x_2^{-1}, \quad x_2^{h_2}={x_2}, \quad x_2^{h_3}=x_2^{-1},\\
x_3^{h_1}=x_3^{-1}, \quad x_3^{h_2}=x_3^{-1}, \quad x_3^{h_3}=x_3.
\end{aligned}
$$
Let $G$ be the semidirect product $N\rtimes H.$ It can be easily checked
that $$(x_1h_1, x_2x_3h_2, x_1x_3h_3, x_1^2x_2h_2, x_2x_3h_3)$$
is a 5-hole in $\Gamma(G).$
\begin{lemma}\label{duemax}
	Let $G$ be a 2-generated finite group and let $m$ be an odd integer, with $m\geq 5.$ Let $X \subseteq G.$ If there exist two maximal subgroups $M_1$ and $M_2$ of $G$ such that  $X \subseteq M_1\cup M_2,$ then $X$ does
	not induce neither an $m$-hole nor an $m$-antihole.
\end{lemma}	
\begin{proof}
	Suppose that $(a_1,\dots,a_m)$ is an $m$-hole induced by $X$. We may assume $a_1\in M_1.$ Since $G=\langle a_i, a_{i+1}\rangle$, it follows $a_i \in M_2 \setminus M_1$ if $i$ is even, $a_i \in M_1 \setminus M_2$ if $i$ is odd. In particular, since $m$ is odd, $G=\langle a_1, a_m \rangle \leq M_1,$ a contradiction.
	Now	suppose that $(a_1,\dots,a_m)$ is an $m$-antihole induced by $X$. Again we may assume $a_1\in M_1.$ If $3\leq i \leq m-1,$ then $G=\langle a_1, a_i\rangle$ implies $a_i \in M_2$ and therefore $m=5,$ otherwise $G=\langle a_3, a_{m-1} \rangle \leq M_2.$ We may exclude this possibility since a $5$-antihole is isomorphic to a $5$-hole.
\end{proof}
\begin{lemma}\label{treprimi}
	Suppose that $G=\left(\langle x \rangle \times \langle y \rangle\right) \rtimes \langle z \rangle$, with $|x|=p_1,$ $|y|=p_2$, $|z|=p_3,$
	where $p_1,p_2,p_3$ are primes.
	If $G$ is 2-generated, then $\Gamma(G)$ is perfect.\end{lemma}
\begin{proof}
	If $G$ is abelian, then the conclusion follows from Theorem \ref{nilperf}. So we may assume $x \notin Z(G).$ Let $m\geq 5$ be an odd integer and suppose that $X\subseteq G$ induces an $m$-hole or an $m$-antihole in  $\Gamma(G).$ 
	
	First we claim that if $y \in Z(G)$, then $p_2 = p_3.$ Indeed assume $y\in Z(G)$ and $p_2\neq p_3$ and let $g=x^iy^jz^k \in G.$ If $|y^jz^k|=p_2p_3,$ then $|g|=p_2p_3$, so $\langle g \rangle$ is the unique maximal subgroup of $G$ containing $g$ and $g \notin X$ by Lemma \ref{unicomax}. But then $X \subseteq M_1\cup M_2,$ with $M_1=\langle x, z \rangle$ and $M_2=\langle x, y \rangle,$ in contradiction with Lemma \ref{duemax}. 
	
	Our second claim is  that $\langle x \rangle$ and $\langle y \rangle$ are not $\langle z \rangle$-isomorphic. This is obvious if $p_1\neq p_2$, otherwise it is a necessary condition for $G$ being 2-generated.
	
	The two previous claims  imply that for every $r,s,u,v \in \mathbb Z$,  $\langle x^ry^s,x^uy^vz \rangle=G$ if and only if  $x^r, y^s \neq 1.$ In particular consider $w=x^ry^s \in \langle x, y \rangle.$ If either $x^r=1$ or $y^s=1,$ then $w$ is an isolated vertex in $\Gamma(G)$. Moreover, the fact that $\langle x \rangle$ and $\langle y \rangle$ are not $\langle z \rangle$-isomorphic implies that
	$\langle x, y \rangle$ is the unique maximal subgroup of $G$ containing
	$u.$ In any case, $u$ cannot be an element of $X$.

	Let $(a_1,\dots,a_m)$ be an $m$-hole or an $m$-antihole in $\Gamma(G)$ induced by $X$. By what we have said above, it is not restrictive to assume $a_i=x^{r_i}y^{s_i}z$ with
	$r_i, s_i \in \mathbb Z$ and we may assume in particular $a_1=z.$ Notice that $\langle x^{r_i}y^{s_i}z,  x^{r_j}y^{s_j}z \rangle =G$ if and only if $r_i \not\equiv r_j\! \mod p_1$ and $s_i \not\equiv s_j\! \mod p_2.$ 
	
	If $(a_1,\dots,a_m)$ is an $m$-hole, then
	$a_1\not\sim a_j$ for any $j\in \{3,\dots,m-1\}$. This implies that either $a_j \in \langle x, z \rangle$ or $a_j \in \langle y, z \rangle.$ On the other hand $a_j \sim a_{j+1}$, so it is not restrictive to assume $s_3\equiv  0\! \mod p_2,\ r_4\equiv  0\! \mod p_1, \dots, s_{m-2}\equiv  0\! \mod p_2,\ r_{m-1}\equiv  0\! \mod p_1.$
	On the other hand, $a_1 \sim a_2$ and $a_1 \sim a_m$ implies
	$r_2, r_m \not\equiv 0 \! \mod p_1$ and $s_2, s_m \not\equiv 0 \! \mod p_2.$
	By Lemma \ref{aiaj}, $r_3\not\equiv 0\! \mod p_1$ and $s_{m-1}\not \equiv 0\! \mod p_2.$ Since $a_2 \not\sim a_{m-1}$ and $a_3\not\sim a_m$, we deduce $s_2 \equiv s_{m-1}\! \mod p_2$ and $r_3 \equiv r_m \! \mod p_1.$ Since $a_2\sim a_3$ and $a_{m-1} \sim a_m,$ it follows $r_2 \not\equiv r_3 \! \mod p_1$ and $s_{m-1} \not\equiv s_m \! \mod p_2,$ but then $r_2 \not\equiv r_m \! \mod p_1$ and $s_{2} \not\equiv s_m
	\! \mod p_2.$ This implies $a_2 \sim a_m$, a contradiction. 
	
	Now suppose that $(a_1,\dots,a_m)$ is an $m$-antihole. We may assume $m\geq 7$ since a 5-antihole is isomorphic to a $5$-hole. From the conditions $a_1 \not\sim a_2$ and $a_1 \not\sim a_m$ it follows that is this not restrictive to assume $s_2=0$ and $r_m=0$. Since $a_i \not\sim a_{i+1}$, it follows
	$$a_1=z,\ a_2=x^{r_2}z,\ a_3=x^{r_2}y^{s_3}z,\ a_4=x^{r_4}y^{s_3}z,
	\	a_5=x^{r_4}y^{s_5}z, \	a_6=x^{r_6}y^{s_5}z \dots$$
	In particular
	$$a_{m-2}=x^{r_{m-3}}y^{s_{m-2}}z,\ a_{m-1}=x^{r_{m-1}}y^{s_{m-2}}z,
	\ a_{m}=y^{s_{m}}z.$$
	From $a_{m-1} \sim a_1,$ it follows $r_{m-1} \not \equiv 0 \! \mod p_1$ and from $a_m \sim a_{m-2},$ it follows $s_m \not \equiv s_{m-2} \! \mod p_2$, however this would imply $a_{m-1} \sim a_m,$ a contradiction.
\end{proof}

\begin{prop}\label{piqu}
	If $G=p_1p_2$ with $p_1, p_2$ primes, then $\Gamma(G)$ is perfect.
\end{prop}
\begin{proof}
	It follows immediately from Lemma \ref{unicomax}.
\end{proof}
\begin{prop}
	If $|G|=p_1p_2p_3,$ where $p_1,$ $p_2$ and $p_3$ are three distinct primes, then $G$ is 2-generated and $\Gamma(G)$ is perfect.
\end{prop}
\begin{proof}By \cite[10.1.10]{rob}, $G$ is 2-metacyclic. We may assume that $G$ is non-abelian, so $G=\langle x \rangle \rtimes \langle y \rangle$, with $y\neq 1$ and $C_{\langle y \rangle}(x)=1.$
	If $|x|$ is the product of two different primes, then the conclusion follows from Lemma \ref{treprimi}. So we may assume that $|x|=p_1.$ 
	Assume that $X$ induces an $m$-hole or an $m$-antihole in $\Gamma(G)$, where $m\geq 5$ is an odd integer. By Lemma \ref{unicomax}, $X$ contains only elements of prime order; moreover an element of order $p_1$ is adjacent in $\Gamma(G)$ only to elements of order $p_2p_3$, so $X$ can contain only
	elements of order $p_2$ or $p_3.$ On the other hand two elements of the same order $p_2$ or $p_3$ are not adjacent in $\Gamma(G)$, and it is easy to see that this implies that $X$ cannot induce neither an $m$-hole nor an
	$m$-antihole.
\end{proof}	
\begin{prop}
	If $|G|=p^2q,$ where $p$ and $q$ are distinct primes, then $\Gamma(G)$ is perfect.
\end{prop}
\begin{proof} First assume that the Sylow $p$-subgroup, say $P$, of $G$ is normal. If $P \cong C_p \times C_p,$ then the conclusion follows from Lemma \ref{treprimi}. If $P\cong C_{p^2},$ then $\frat(G)$ has order $p$, so $G/\frat(G)$ is perfect by Proposition \ref{piqu} and consequently, by
	Lemma \ref{qfra}, also $\Gamma(G)$ is perfect. So we may assume that the
	Sylow $p$-subgroups are not normal, which implies that the Sylow $q$-subgroup, say $Q,$ is normal.
	Either $G\cong (C_p \times C_q) \rtimes C_p$  or $G\cong C_q \rtimes C_{p^2}.$ In the first case the conclusion follows again from Lemma \ref{treprimi}. In the second
	case the only non-trivial elements of $G$ that are contained in at least two different maximal subgroups are those of order $p,$ so, by Lemma
	\ref{unicomax}, if $X\subseteq G$ induces an $m$-hole or an $m$-antihole,
	with $m \geq 5$ an odd integer, then $X$ contains only elements of order $p.$ However two elements of order $p$ are not adjacent in $\Gamma(G)$, so we reached a contradiction.
\end{proof}

\subsection{The symmetric and alternating group}
In this subsection we determine for which values of $n$, the symmetric and alternating group of degree $n$ have perfect generating graph.
In the proofs we will need the following elementary lemmas:
\begin{lemma}\label{prim1}
	Let $H\leq S_n$ be a transitive permutation group. If $\sigma\in H$ is a $(n-1)$-cycle, then $H$ is primitive.
\end{lemma}
\begin{proof}
	We may assume without loss of generality that the fixed point of $\sigma$ is $1$. Suppose $H$ is imprimitive. Let $B$ be the imprimitivity block which contains $1$, then $B^\sigma=B$ since $1^\sigma=1$. By the imprimitivity assumption, there exists $1\neq i\in B$. But then $i,i^\sigma,i^{\sigma^2},\dots i^{\sigma^{n-2}}$ are all distinct elements, so $B=\{1,\dots,n\}$, a contradiction.
\end{proof}
\begin{lemma}\label{prim2}
	Let $n\geq3$ be an odd natural number and $H\leq S_n$ be a transitive permutation group. If $\sigma\in
	H$ is a $(n-2)$-cycle, then $H$ is primitive.
\end{lemma}
\begin{proof}
	Suppose, without loss of generality, that the fixed points of $\sigma$ are $1$ and $2$. Suppose $H$ is imprimitive. As in the proof of the previous Lemma, take $B$ to be the block containing $1$. Since $n$ is odd, $|B|\geq 3$, so there is at least an element $i$ in $B\setminus \{1,2\}$. Arguing as in the proof of the previous lemma, we obtain that $|B|\geq n-1$ and, since $|B|$ divides $n$, we conclude $B=\{1,\dots,n\}$, a contradiction.
\end{proof}
\begin{thm}
	$\rg(S_n)$ is perfect \ifa $n \leq 4$.
\end{thm}
\begin{proof} We distinguish the different possible values of $n.$
	
	\noindent a) $n=2.$ In this case $\Gamma(S_2)\cong K_2$ is perfect.
	
	\noindent b) $n=3$. Since $S_3\cong D_3$, it follows from Theorem \ref{diedrale} that $\Gamma(S_3)$ is perfect.
	
	\noindent c) $n=4$. Suppose there is an $m$-hole $(a_1,\dots,a_m)$ in $\Gamma(S_4),$ with $m\geq 5.$ Two consecutive vertices $a_i$ and $a_{i+1}$ are adjacent and therefore they cannot belong both to $A_4$. Since $m$ is odd, there must be two consecutive vertices which are in $S_4\setminus
	A_4$. Since two elements of order $2$ do not generate the group, one of these two vertices should be a $4$-cycle. However a $4$-cycle is contained
	in a unique maximal subgroup, so we have a contradiction by Lemma \ref{unicomax}. Suppose now that there is an $m$-antihole
	$(a_1,\dots,a_m)$ in $\Gamma(S_4)$, with $m\geq7$. Since $4$-cycles cannot occur in an $m$-antihole and elements of the Klein subgroup cannot generate with another element, the vertices of the antihole can be only transpositions and $3$-cycles. There are at most two $3$-cycles among the vertices of the antihole. Indeed if we pick three elements in an $m$-antihole, at least two of them are adjacent but two $3$-cycles do not generate $S_4$. So, at least $m-2$ of the vertices of the antihole $(a_1,\dots,a_m)$
	are transpositions. Since two transpositions do not generate $S_4$, we have a contradiction.
	We conclude that $\Gamma(S_4)$ is perfect.
	
	\noindent d) $n=5,6,7$. It can be easily checked that the following are $5$-holes in $\Gamma(S_n):$
	\[
	\begin{aligned}
	\rg(S_5):&\qquad ((1,2,3,4,5),(2,4),(1,2,3,5,4),(2,4,5,3),(1,2,4,5));\\
	\rg(S_6):&\qquad ((1,3,2,4),(3,4,6,5),(1,2,3,4,5),(1,3,4,6),(2,3,4,5,6)); \\
	\rg(S_7):&\qquad ((1,5,4,7,2,3),(2,6,5,7,3,4),(1,2,3,4,5,7,6),(4,5),(1,2,3,4,5,6,7)).
	\end{aligned}
	\]
	
	\noindent  e) $n\geq8$ even. In this case we claim that
	\[
	\begin{aligned}
	&a_1=(1, \dots, n-2)\\
	&a_2=(3, \dots, n)\\
	&a_3=(1, \dots, n-1)\\
	&a_4=(1, 3, 4, n)\\
	&a_5=(2, \dots, n)
	\end{aligned}
	\]
	is a $5$-hole in $\rg(S_n)$.
	
	Notice that
	$\gen{a_1,a_3}$, $\gen{a_1,a_4}$, $\gen{a_2,a_4}$, $\gen{a_2,a_5}$ are intransitive subgroups and $\gen{a_3,a_5}\leq A_n$, so the pairs of corresponding vertices are not joined by an edge.
	Since $a_3$ and $a_5$ are $(n-1)$-cycles, the  transitive subgroups $\gen{a_2,a_3}$,  $\gen{a_3,a_4}$, $\gen{a_4,a_5}$, $\gen{a_5,a_1}$ are also
	primitive by Lemma \ref{prim1}. Let us now prove that also the transitive
	subgroup $\gen{a_1,a_2}$ is primitive. Let  $B$ be an imprimitive block which contains $1$. Clearly $B^{a_2}=B$. If $B\cap\{3,\dots,n\} \neq \emptyset,$ then  $\{1,3,\dots,n\}\subseteq B$, a contradiction. So $B=\{1,2\}$, but then $B\cap B^{a_1}=\{2\}$, another contradiction. Moreover
	\[
	\begin{aligned}
	a_1a_2^{-1}&=(1,2,n,n\!-\!1,n\!-\!2)&\in\gen{a_1,a_2}\\
	a_3a_2^{-1}&=(1,2,n,n\!-\!1)&\in\gen{a_2,a_3}\\
	a_4&=(1,3,4,n)&\in\gen{a_3,a_4}\\
	a_4&=(1,3,4,n)&\in\gen{a_4,a_5}\\
	a_1a_5^{-1}&=(1,n,n\!-\!1,n\!-\!2)&\in\gen{a_5,a_1}
	\end{aligned}
	\]
	But then, by \cite[Corollary 1.3]{Jo}, the  five subgroups
	$\gen{a_1,a_2},$ $\gen{a_2,a_3},$ $\gen{a_3,a_4},$ $\gen{a_4,a_5}$ and $\gen{a_5,a_1}$
	contain $A_n$. Since they contain elements outside $A_n$, they must be
	the whole $S_n$, so the pairs of corresponding vertices are joined by an edge.
	
	\noindent f) $n\geq 9$ odd. In this case we claim that
	\[
	\begin{aligned}
	&a_1=(1, \dots, n-3)\\
	&a_2=(4, \dots, n)\\
	&a_3=(1, \dots, n-2)\\
	&a_4=(1, 2, 4, 5, n\!-\!1, n\!-\!2)\\
	&a_5=(3, \dots, n)
	\end{aligned}
	\]
	is a $5$-hole in $\rg(S_n)$.\\
	As in the discussion of the previous case, $\gen{a_1,a_3}$, $\gen{a_1,a_4}$, $\gen{a_2,a_4}$, $\gen{a_2,a_5}$ are intransitive subgroups and $\gen{a_3,a_5}\leq A_n$, so the pairs  of corresponding vertices are not joined by an edge.
	Since $a_3$ and $a_5$ are $(n-2)$-cycles, the transivite subgroups $\gen{a_2,a_3}$, $\gen{a_3,a_4}$, $\gen{a_4,a_5}$, $\gen{a_5,a_1}$ are also primitive by Lemma \ref{prim2}. Let us now prove that also the transitive subgroup $\gen{a_1,a_2}$ is primitive. Let $B$ be an imprimitivity  block which contains $1$, so that $B^{a_2}=B$. If $B\cap \{4,\dots,n\} \neq \emptyset,$ then $\{1,4,\dots,n\}\subseteq B$, a contradiction. Since $|B|\geq3$, the only possibility is $B=\{1,2,3\}$, but this leads to a contradiction since $B\cap B^{a_1}=\{2,3\}$.
	Moreover
	\[
	\begin{aligned}
	a_1&=(1, \dots, n-3)&\in\gen{a_1,a_2}\\
	a_3a_2^{-1}&=(1,2,3,n,n\!-\!1,n\!-\!2)&\in\gen{a_2,a_3}\\
	a_4&=(1, 2, 4, 5, n\!-\!1, n\!-\!2)&\in\gen{a_3,a_4}\\
	a_4&=(1, 2, 4, 5, n\!-\!1, n\!-\!2)&\in\gen{a_4,a_5}\\
	a_1&=(1, \dots, n-3)&\in\gen{a_5,a_1}
	\end{aligned}
	\]
	and, as in the previous case, we deduce from \cite[Corollary 1.3]{Jo}
	that
	$\gen{a_1,a_2} = \gen{a_2,a_3} = \gen{a_3,a_4} = \gen{a_4,a_5}= \gen{a_5,a_1}= S_n.$
\end{proof}

\begin{thm}\label{alterni}
	$\rg(A_n)$ is perfect \ifa $n\leq 4$.
\end{thm}

\begin{proof} We distinguish the different possible values of $n.$
	
	\noindent a) $n=3.$ In this case $\Gamma(A_3)\cong K_3$ is perfect.
	
	\noindent b)  $n=4$. Let $m$ be an odd positive integer, with $m\geq 5.$
	In an $m$-hole or in an $m$-antihole, at least one vertex should be a $3$-cycle, since in a pair of generators one should be outside the Klein subgroup. However a $3$-cycle is contained in a unique maximal subgroup, so we conclude using  Lemma
	\ref{unicomax} that there is neither an $m$-hole nor  an $m$-antihole. 
	
	\noindent c) $n=6$. In this case $\Gamma(A_6)$ contains the following
	$5$-hole:
	\[
	((1,2,3,4,5),(1,3)(5,6),(1,2,4,5,6),(1,4,2,3,5),(1,2,6)).\\
	\]
	
	\noindent d)  $n\geq 5$ odd. In this case we claim that
	\[
	\begin{aligned}
	&a_1=(1,2,3,6 \dots, n)\\
	&a_2=(2,4,5,6 \dots, n)\\
	&a_3=(1,3,5,6 \dots, n)\\
	&a_4=(2,3,4,6 \dots, n)\\
	&a_5=(1,4,5,6 \dots, n)
	\end{aligned}
	\]
	is a $5$-hole in $\rg(A_n)$.	The cases $n=5,7,9$ can be easily checked by hand, so we assume $n\geq11$. Notice that
	$\gen{a_1,a_3}$, $\gen{a_1,a_4}$, $\gen{a_2,a_4}$, $\gen{a_2,a_5}$ and
	$\gen{a_3,a_5}$ are intransitive subgroups, so the pair  of corresponding vertices are not joined by an edge.
	Since $a_1,\dots, a_5$ are $(n-2)$-cycles, the transitive subgroups $\gen{a_1,a_2}$, $\gen{a_2,a_3}$, $\gen{a_3,a_4}$, $\gen{a_4,a_5}$ and $\gen{a_5,a_1}$ are also primitive from Lemma \ref{prim2}. Moreover
	\[
	\begin{aligned}
	a_1^2a_2^{-2}&=(1,3,5,2,4,n,n\!-\!1)&\in\gen{a_1,a_2}\\
	a_2a_3^{-1}&=(1,n,2,4,3)&\in\gen{a_2,a_3}\\
	a_3^2a_4^{-2}&=(1, 5, 4, 2, n\!-\!1)&\in\gen{a_3,a_4}\\
	a_4^2a_5^{-2}&=(1,n\!-\!1, 2,n,3, 4, 5)&\in\gen{a_4,a_5}\\
	a_5a_1^{-1}&=(1,4,5,3,2)&\in\gen{a_5,a_1}
	\end{aligned}
	\]
	and we can use \cite[Corollary 1.3]{Jo} to conclude that
	$\gen{a_1,a_2} = \gen{a_2,a_3} = \gen{a_3,a_4} = \gen{a_4,a_5}=
	\gen{a_5,a_1}= A_n.$
	
	\noindent e) $n\geq 8$ even. In this case we claim that
	\[
	\begin{aligned}
	&a_1=(1,2,3,4,5,9 \dots, n)\\
	&a_2=(1,3,6,7,8,9 \dots, n)\\
	&a_3=(2,7,8,4,5,9 \dots, n)\\
	&a_4=(1,6,3,4,5,9 \dots, n)\\
	&a_5=(1,2,6,7,8,9 \dots, n)
	\end{aligned}
	\]
	is a $5$-hole in $\rg(A_n)$. The subgroups
	$\gen{a_1,a_3}$, $\gen{a_1,a_4}$, $\gen{a_2,a_4}$, $\gen{a_2,a_5}$ and $\gen{a_3,a_5}$  are intransitive, so the pair  of corresponding vertices
	are not joined by an edge.
	We prove that the transitive subgroup $\gen{a_1,a_2}$ is primitive (a similar argument works for the subgroups $\gen{a_2,a_3}$, $\gen{a_3,a_4}$,
	$\gen{a_4,a_5}$ and $\gen{a_5,a_1}$). Suppose it is imprimitive. Let $B$ be an imprimitive block containing $6$, so that $B^{a_1}=B$. We must have  $B\subseteq \{6,7,8\}$, otherwise we would have $\{1,2,3,4,5,6,9,\dots, n\}\subseteq B$. Moreover, since $B\cap B^{a_2}=\{7,8\}$, $B \neq \{6,7,8\}$, so either  $B=\{6,7\}$ or $B=\{6,8\}$. In the first case  the block which contains $8$, contains also an element different from $6,7,8$ and we get a contradiction as before; the same happens in the second case considering the block which contains $7$. Since $a_1,\dots,a_5$ are $(n-3)$-cycles, we can conclude, using \cite[Corollary 1.3]{Jo}, that 	$\gen{a_1,a_2} = \gen{a_2,a_3} = \gen{a_3,a_4} = \gen{a_4,a_5}= \gen{a_5,a_1}= A_n.$
\end{proof}

\subsection{Other simple groups}\label{semp}

In this subsection we prove that Conjecture \ref{cong2} is true also
when $G\cong \psl_2(q)$, $G\cong {^2B}_2(q)$ or $G\cong {^2G}_2(q).$
We need the following elementary observation.

\begin{lemma}\label{stabilizer}
	Let $G$ be a permutation group on the set $\Omega$. Let $\omega_1,\dots,\omega_5\in\Omega$ \st $\stab_G(\omega_i)<G$ for $i=1,\dots,5$. Let
	\begin{align*}
	&a\in \stab_G(\omega_1)\cap\stab_G(\omega_2),\\
	&b\in \stab_G(\omega_3)\cap\stab_G(\omega_4),\\
	&c\in \stab_G(\omega_5)\cap\stab_G(\omega_1),\\
	&d\in \stab_G(\omega_2)\cap\stab_G(\omega_3),\\
	&e\in \stab_G(\omega_4)\cap\stab_G(\omega_5).
	\end{align*}
	If $\gen{a,b}=\gen{b,c}=\gen{c,d}=\gen{d,e}=\gen{e,a}=G$, then
	$(a,b,c,d,e)$ is a $5$-hole in $\rg(G)$.
\end{lemma}

\begin{prop}
	Let $G=\psl_2(q),$ with $q> 3$. Then $\Gamma(G)$ contains a 5-hole.
\end{prop}
\begin{proof} We may assume $q\notin \{4, 5,9\},$ since
	$\psl_2(4)\cong\psl_2(5)\cong A_5$ and $\psl_2(9)\cong A_6$.
	The group $G$  has a faithful 2-transitive action on the $q+1$ points of the 1-dimensional projective space $\pg(1,q)$
	over the field  $\mathbb{F}_q$ with $q$ elements.
	Let $A,B,C,D$ be four distinct points of $\pg(1,q)$. The subgroups $H=\stab_G(A)\cap\stab_G(B)$ and $K=\stab_G(C)\cap\stab_G(D)$ are cyclic of order $u=(q-1)/(q-1,2)$. Notice that $H\neq K$, since the only element of $G$ which fixes three distinct points is the identity. The list of the maximal subgroups of $G$ is well-known (see for example \cite[Table 8.1]{max_simple}). In particular if  $q \notin\{7,11\}$, then no maximal subgroup of $G$ contains two distinct cyclic subgroups of order $u$. This implies $G=\langle H,K \rangle$. But then we can use Lemma \ref{stabilizer} to conclude.
	
	\noindent	If $q=7$, then $G\leq \perm(8)$ and the following is a 5-hole in $\Gamma(G):$
	$$((2,3,4)(5,8,7), (1,4,5)(3,7,6),
	(2,7,8)(3,6,5), (1,2,4)(6,7,8),
	(1,2,5,7)(3,8,6,4))$$
	If $q=11,$ then $G\leq \perm(12)$ and the following is
	a 5-hole in $\Gamma(G):$
	$$
	((3,9,5,11,7)(4,10,6,12,8),
	(1,6,3,4,12)(2,11,9,10,7),
	(1,3,8,5,4)(6,7,9,12,10),$$
	$$
	(2,12,11,8,3)(4,7,9,10,6),
	(1,9,6,7,5)(2,4,12,3,10)).
	\qedhere$$
\end{proof}

\begin{prop}\label{suzuki}
	Let $q=2^{2n+1}$ with $n\geq 1.$ If $G={^2B}_2(q)$ is a Suzuki group, then $\rg(G)$ contains a 5-hole.
\end{prop}

\begin{proof}
	The group $G$ has a faithful 2-transitive action on an ovoid $\Omega$ in 4-dimensional symplectic geometry. Moreover the following are up to conjugacy
	the maximal subgroups of $G$  (see for example \cite[Theorem 7.3.3]{max_simple}):
	\begin{enumerate}
		\item the stabilizer of $\omega \in \Omega$ (the Borel subgroup of order 					$q^2(q-1)$);
		\item the dihedral group of order $2(q-1)$;
		\item $C_{q+\sqrt{2q}+1}: C_4$;
		\item $C_{q-\sqrt{2q}+1}: C_4$;
		\item ${^2B}_2(q_0)$, where $q=q_0^r$, $r$ is prime and $q_0>2$.
	\end{enumerate}
	
	If $\omega_i$ and $\omega_j$ are distinct elements of $\Omega,$ then $\stab_G(\omega_i)\cap\stab_G(\omega_j)$ is cyclic of order $q-1$. Let $x$ be a generator of this cyclic group. Let now $\omega_l$ and $\omega_k$ in $\Omega$ \st $\omega_i,\omega_j,\omega_k,\omega_l$ are all distinct, and call $y$ the generator of $\stab_G(\omega_k)\cap\stab_G(\omega_l)$. Since the only element fixing three points is the identity, we have that $\gen{x}\neq \gen{y}$. Consider the subgroup $H:=\gen{x,y}$. If $H$ is a proper subgroup, it is contained in a maximal subgroup. However $H$ cannot be contained in subgroups of type $(3)$, $(4)$ and $(5)$, since they do not contain elements of order $q-1$. Since $\gen{x}\neq\gen{y}$ we can exclude also the fact that $H$ is contained in the subgroups of type $(2)$, since the dihedral group contains a unique subgroup of order $q-1$. Finally, if $H\leq\stab_G(\omega)$, for some $\omega\in\Omega$, at least one of $x$ and $y$ should fix three different points, which is impossible, therefore $H=G$ and using Lemma \ref{stabilizer} we can build a $5$-hole.
\end{proof}

\begin{prop}\label{ree}
	Let $q:=3^{2n+1}$ with $n\geq 1.$ If $G:={^2G_2}(q)$, then $\Gamma(G)$ contains a 5-hole.
\end{prop}
\begin{proof}
	The groups $G$ has a faithful 2-transitive action on an ovoid $\Omega$ in 7-dimensional orthogonal geometry. Moreover the following are up to conjugacy
	the maximal subgroups of $G$ (see for example \cite[Table 8.43]{max_simple}):
	\begin{enumerate}
		\item the stabilizer of $\omega \in \Omega$ (the Borel subgroup of order 					$q^3(q-1)$);
		\item the centralizer of an involution, which is isomorphic to $C_2\times\psl_2(q)$;
		\item the normalizer of a four-group, which is isomorphic to
		$(2^2 \times D_{(q+1)/4}):3;$
		\item $C_{q+\sqrt{3q}+1}: C_6$;
		\item $C_{q-\sqrt{3q}+1}: C_6$;
		\item ${^2G}_2(q_0)$, where $q=q_0^r$ and $r$ prime.
	\end{enumerate}
	The intersection of two different point-stabilizers is cyclic with order $q-1$. Moreover any involution $t$ in $G$ fixes precisely $q + 1$ points in $\Omega$, and the set of these $q+1$ elements is called the block of $t.$ Any two blocks can intersect in at most 1 point and any two points are pointwise fixed by a unique involution.
	Choose $\omega_1,\dots,\omega_5\in\Omega$ all distinct in the
	following way: $\omega_1$, $\omega_2$ and $\omega_3$ are  chosen randomly and let $\Omega_{i,j}$ be the unique block which contains $\omega_i$ and $\omega_j$. Since $|\Omega|=q^3+1$ and a block has cardinality $q+1$, it is possible to choose $\omega_4\in\Omega\setminus\Omega_{2,3}$ and $\omega_5\in\Omega\setminus(\Omega_{3,4}\cup\Omega_{1,2})$. Since the block containing two elements is unique, we have that four of these five elements never belong to the same block.
	Let $\omega_i,\omega_j,\omega_k,\omega_l$ be four of these five elements. Let $x$ a generator of $\stab_G(\omega_i)\cap\stab_G(\omega_j)$ and $y$ a generator of $\stab_G(\omega_l)\cap\stab_G(\omega_k)$ and consider $H:=\gen{x,y}$. The subgroup $H$ cannot be contained in maximal subgroups of type $(3)$, $(4)$, $(5)$ and $(6)$, since these maximal subgroups do not contain elements of order $q-1$. There are no elements in $G$ of order $q-1$ which fix three distinct elements on $\Omega$, so $H$ is not contained in maximal subgroups of type $(1)$. Therefore, if $H<G$ is proper, then $H\leq C_G(t)$ for a suitable involution $t$ of $G$. This occurs when $t$ is contained in the intersection of the four stabilizers of $\omega_i,\omega_j,\omega_k,\omega_l$ 
	but in this case 
	$\omega_i,\omega_j,\omega_k,\omega_l$ belong to the same block, against our  choice of $\omega_1,\dots,\omega_5$. So $H=G$ and we may conclude by applying Lemma \ref{stabilizer}.
\end{proof}
\subsection{Small groups}
The alternating group $A_5$ of degree 5 is the smallest 2-generated  finite group whose generating graph is not perfect. More precisely
the following holds.
\begin{thm}\label{mina5}
	Let $G$ be a 2-generated finite group, with $|G|\leq 60.$ Then $\Gamma(G)$ is perfect if and only if $G\neq A_5.$
\end{thm}
\begin{proof}
	By Theorem \ref{alterni}, we have only to prove that if $|G|\leq 60$ and
	$G\neq A_5,$ then $\Gamma(G)$ is perfect. By Theorem \ref{nilperf}, $C_{30}\times C_6$ is the smallest 2-generated finite nilpotent group whose generating graph is not perfect. So we may assume that $G$ is not nilpotent, and by the results in subsection \ref{piccoli} we may exclude $|G|\in \{pq, pqr, p^2q\}$ with $p, q, r$ different primes. Hence $|G|\in \{24,36,40,48,54,56,60\}.$ This requires a case by case analysis. As an example we consider
	$G \cong C_6 \times D_5$, which is the case that requires more attention.
	The other cases can be discussed with similar, but in general shorter, arguments. Let $m\geq5$ be odd. Set $\gen{c}=C_6$ and $\gen{\rho,\iota}=D_5$, with $\rho$ a rotation of order $5$ and $\iota$ a reflection. Since
	$C_2\times C_2$ is an epimorphic image of $G$, it follows from Lemma \ref{c2c2} that $\rg(G)$ does not contain $m$-antiholes with $m\geq7$, so we have to check only the non-existence of  $m$-holes. To prove this, we need the list of maximal subgroups of $G$:
	\begin{itemize}
		\item $M_1= \gen{c^2, \iota, \rho}$;
		\item $M_2= \gen{c^3, \iota, \rho}$;
		\item $M_3=  \gen{c,\rho}$;
		\item $M_4 =\gen{\rho,c\iota}$;
		\item $M_{5+\alpha}= \gen{c,\rho^\alpha\iota}$ with $\alpha \in  \{0,1,2,3,4\}.$
		
	\end{itemize}
	Suppose $(a_1,\dots,a_m)$ is a $m$-hole. Consider the two projections $\pi_1: G\to \langle c\rangle,$ $\pi_2: G \to \gen{\rho,\iota}.$ There exists $i \in \{1,\dots,m\},$ such that $\langle \pi_1(a_i)\rangle=\langle c\rangle.$ Otherwise $|\pi_1(a_j)|\in \{2,3\}$ for every $1\leq j \leq m,$ and, since $m$ is odd, there would exist two consecutive vertices $a_k$ and $a_{k+1}$ with $|\pi_1(a_k)|=|\pi_1(a_{k+1})|=t \in \{2,3\}$, and consequently $G=\gen{a_k,a_{k+1}}\leq \gen {c^{6/t}}\times D_5.$  So we may assume without loss of generality $\pi_1(a_1)=c.$ It must be $\pi_2(a_1)\neq 1$ (otherwise $a_1$ would be an isolated vertex of $\Gamma(G).$ Moreover $\pi_2(a_1) \notin \gen{\rho},$ otherwise
	$M_3$ would be the unique maximal subgroup of $G$ containing $a_1,$ against Lemma \ref{unicomax}. So we may assume $a_1=c\iota$. Let $3\leq j\leq m-1.$ Since $\gen {a_1, a_j}\neq G$ and $M_4$,   $M_5$ are the unique maximal subgroups of $G$ containing $a_1,$ it follows
	$a_j \in M_4\cup M_5$. Two consecutive vertices of $(a_1,\dots,a_m)$ generate $G$, so they cannot belong to the same maximal subgroup. Hence we can label the vertices of the $m$-hole so that $a_3 \in M_5$ (and consequently $a_4\in M_4)$. So $a_3=c^\alpha\iota$ with
	$\alpha \in \{0,1,2,3,4,5\}.$
	Moreover $\alpha\neq3$ (otherwise $\gen{a_3,a_4}\leq M_4$) and $\alpha \notin \{1,5\}$ (since, by Lemma  \ref{aiaj}, $\gen{a_3}\neq\gen{a_1}$), and so we have $\alpha=\pm2$.
	Notice that $M_1$ and $M_5$ are the unique maximal subgroups of $G$ containing $a_3,$ so $a_m \in M_1 \cup M_5.$ On the other hand, from $\gen{a_1, a_m}=G$ and $a_1\in M_5,$ it follows $a_m \notin M_5$ and therefore $a_m \in M_1.$
	Now let $a_2=c^x\rho^y\iota^z$, with $0\leq x \leq 5,$ $0 \leq y \leq 4$ and $0\leq z\leq 1.$
	We have  $x\notin \{0,2,4\}$, otherwise $\gen{a_2,a_3}\leq M_1.$ If $x\in
	\{1,5\}$, then $z=1$ (otherwise $a_2$ would have order 30 and consequently would be contained in a unique maximal subgroup), but this would imply $\gen{a_1,a_2}\leq M_4,$ a contradiction.  So it must be $x=3$. If $z=1$, then again $\gen{a_1,a_2}\leq M_4$, a contradiction, so $z=0$ and therefore $a_2=c^3\rho^y$ with $y\neq0$. In particular  $M_2$ and $M_3$ are the unique maximal subgroups of $G$ containing $a_2$, and,  since
	$a_m$ and $a_2$ are not adjacent, $a_m \in M_2\cup M_3$. We have already proved that $a_m\in M_1$ so
	$a_m \in (M_1\cap M_2) \cup (M_1\cap M_3).$ Since $M_1\cap M_3 \leq M_4$ and $a_1 \in M_4,$ if $a_m \in M_1\cap M_3,$ then $G=\langle a_1, a_m \rangle \leq M_4,$ a contradiction.
	So $a_m \in M_1\cap M_2=\langle \rho, \iota \rangle,$ and consequently we may assume $a_{m-1}=c\rho^s\iota^t$. We have $t=1$, otherwise $a_{m-1}$ is contained in a unique maximal subgroup. Then
	$
	\gen{a_2,a_{m-1}}=\gen{c\rho^s\iota,c^3\rho^y}=\gen{c\rho^s\iota,c^3,\rho}=\gen{c\iota,c^3,\rho}=\gen{c,\iota,\rho}=G,
	$
	a	contradiction.
\end{proof}

\section{Other forbidden graphs}
\begin{prop}
	Let $G$ be a non-trivial finite group. Then $G$ does not contain an induced graph isomorphic to $P_3$ if and only if either $G\cong C_2\times C_2$
	or $G\cong C_p$ for some prime $p.$
\end{prop}
\begin{proof}
	Suppose that $G$ satisfies the following property: $(*)$ there exist $a,
	b \in G$ such that $G=\langle a, b \rangle,$ $G\neq \langle a \rangle,$
	$G\neq \langle b \rangle$ and $a \neq a^{-1}.$
	Then $(a,b,a^{-1})$ is a three-vertex path in $\Gamma(G).$ Assume that $G=\langle a, b\rangle$ is not cyclic. If $G$ is not a dihedral group, then  $(|a|,|b|)\neq (2,2)$. If $G\cong D_n$ is a dihedral group of order $2n,$ then we may choose $a, b$ such that $(|a|,|b|)=(n,2).$
	So if $G$ is not cyclic, then either $G$ satisfies  $(*)$  or $G\cong C_2
	\times C_2.$ In this latter case, assume that $(x_1,x_2,x_3)$ is a three-vertex path in $\Gamma(G)$: then $x_i \neq 1$ for any $i\in \{1,2,3\},$ but then $\{x_1,x_2,x_3\}$ induces a complete graph $K_3.$ So we may assume that $G=\langle x \rangle \cong C_n.$ If $n=rs$ and $(r,s)=1$, then
	$G=\langle x^r, x^s \rangle$ and $(|x^r|,|x^s|)=(s,r) \neq (2,2)$ so $G$ satisfies  $(*)$. If $n=p^t$ with $p$ a prime and $t\geq 2,$ then
	$(1,x,x^p)$ is a  three-vertex path in $\Gamma(G).$ If $n=p$ is a prime
	and $x_1,x_2,x_3$ are three distinct elements of $G,$ then $\{x_1,x_2,x_3\}$ induces a complete graph $K_3.$
\end{proof}

\begin{lemma}\label{quattro}
	Let $p$ be a prime and assume that either $G$ is a cyclic $p$-group or $|G|=2p.$ Suppose that the subgraph of $\Gamma(G)$ induced by four
	distinct non-isolated vertices contains at least one edge. Then at least one of the four vertices is adjacent to all the others.
\end{lemma}
\begin{proof}
	Assume that $X=\{g_1,g_2,g_3,g_4\}$ induces a non empty-edges subgraph
	of $\Gamma(G).$ If $G\cong C_{p^n}$ is cyclic of prime power order, then there exists $i$ with $|g_i|=p^n$ (otherwise all the elements of $X$ belong to the unique maximal subgroup of $G$). But then $g_i$ is adjacent to $g_j$ whenever $j\neq i$. Assume $|G|=2p$. If $X$  contains an element of order $2p$ then this element generates $G$ so it is adjacent to all the others. Moreover
	we cannot have $|g_j|=p$ for every $j\in \{1,\dots,4\}$, since all the
	elements of order $p$ belong to the same maximal subgroup. Thus there exists $g_i\in X$ with $|g_i|=2$, but again this implies that $g_i$ is adjacent to $g_j$ whenever $j\neq i$.
\end{proof}

A graph is chordal if it contains no induced cycle of length greater then
3. A graph is called split if its vertex set is the disjoint union of two
subsets $A$ and $B$ so that $A$ induces a complete graph and $B$ induces
an empty graph.

\begin{prop}Let $G$ be a 2-generated finite group. Then the following conditions are equivalent.
	\begin{enumerate}
		\item $\Gamma(G)$ is split.
		\item $\Gamma(G)$ is chordal.
		\item $\Gamma(G)$ is $C_4$-free.
		\item Either $G$ is a cyclic $p$-group or $|G|=2p$ for some prime $p.$
\end{enumerate}\end{prop}
\begin{proof} Clearly (2) implies (3).
	
	Assume that (3) holds. If there exist $a, b \in G$ so that $G=\langle a,b \rangle$, $\langle a \rangle \neq G$, $\langle b \rangle \neq G$, $|a| \neq 2,$ $|b| \neq 2$, then the subgraph of $\Gamma(G)$ induced by $\{a,b,a^{-1},b^{-1}\}$ is a four-vertex cycle. If $G$ is cyclic of order $n$, then we can find $a, b$ with the previous properties except when $n$ is a prime-power or $n=2p$ with $p$
	a prime. So we may assume that $G$ is non-cyclic and
	$G=\langle a, b\rangle$ with $|a|=2.$ Moreover either $b$ or $ab$ has order 2, otherwise $(b,ab)$ is a generating pair with the previous properties. Hence $G=\langle a, b \rangle=\langle a, ab\rangle$ can be generated by two involutions, so $G$ is isomorphic to a dihedral  group $D_{n}$ of order $2n$ and we may assume $|b|=n.$ If $n$ is not a prime and $p$ is a prime divisor of $n,$ then the subgroup of $\Gamma(G)$ induced
	by $\{a,b,ab^p,b^{-1}\}$ is a four-vertex cycle. 
	
	It follows from Lemma \ref{quattro} that (4) implies (2).
	
	It was shown in \cite{fh} that a graph is split if and only if it does not have an induced
	subgraph isomorphic to one of the three forbidden graphs, $C_4,$ $C_5$ or $2K_2.$ In particular (1) implies (3) and we may immediately deduce from Lemma \ref{quattro} that (4) implies (1).
\end{proof}

\end{document}